\documentclass[oneside,english]{amsart}
\usepackage[T1]{fontenc}
\usepackage[latin9]{inputenc}
\usepackage{amsthm}
\usepackage{amssymb}
\usepackage{esint}

\makeatletter
\numberwithin{equation}{section}
\numberwithin{figure}{section}
\theoremstyle{plain}
\newtheorem{thm}{\protect\theoremname}
  \theoremstyle{definition}
  \newtheorem{defn}[thm]{\protect\definitionname}
  \theoremstyle{plain}
  \newtheorem{lem}[thm]{\protect\lemmaname}
  \theoremstyle{remark}
  \newtheorem{rem}[thm]{\protect\remarkname}
  \theoremstyle{plain}
  \newtheorem{prop}[thm]{\protect\propositionname}

\makeatother

\usepackage{babel}
  \providecommand{\definitionname}{Definition}
  \providecommand{\lemmaname}{Lemma}
  \providecommand{\propositionname}{Proposition}
  \providecommand{\remarkname}{Remark}
\providecommand{\theoremname}{Theorem}

\begin{document}

\title{Factorial decay of iterated rough integrals }

\author{Horatio Boedihardjo}

\address{Department of Mathematics and Statistics, University of Reading,
Reading, RG6 6AX, UK. }
\begin{abstract}
In this complementary note to \cite{BRP15} (arXiv:1501.05641), we
provide an alternative proof for the factorial decay estimate of iterated
integrals for geometric rough paths without using the neoclassical
inequality. This note intends to aid the readers on the proof in \cite{BRP15}
which works also for branched rough paths.  Just as in \cite{BRP15},
the proof here is an extension of Lyons 94' \cite{Lyons94} from Young's
integration to geometric rough paths. 
\end{abstract}

\thanks{We would like to thank the anonymous referee for \cite{BRP15} for
the useful comments. }

\maketitle
Let $X$ be a path in a Banach space $E$ and $A$ be a linear map
$E\rightarrow L(F,F)$, where $F$ is another Banach space. The controlled
differential equation 
\begin{equation}
\mathrm{d}Y_{t}=A(\mathrm{d}X_{t})(Y_{t})\label{eq:controlled de}
\end{equation}
has an explicit series expansion of the form 
\begin{equation}
Y_{t}=\sum_{k=0}^{\infty}\int_{0<s_{n}<\ldots<s_{1}<t}A(\mathrm{d}X_{s_{1}})\ldots A(\mathrm{d}X_{s_{n}})Y_{0}\label{eq:series expansion}
\end{equation}
as long as the series converges. As Lyons noted in \cite{Lyons98},
a first step to make sense of (\ref{eq:controlled de}) is to make
sense of the iterated integrals 
\[
\int_{s<s_{1}<\ldots<s_{n}<t}\mathrm{d}X_{s_{1}}\otimes\ldots\otimes\mathrm{d}X_{s_{n}}
\]
and to prove an estimate for the iterated integral that ensures the
series (\ref{eq:series expansion}) converges. The first result in
Lyons' original work was in fact aimed to resolve these two questions.
To recall Lyons' result, we will use the notation 
\[
\triangle_{n}=\{(s_{1},s_{2},\ldots,s_{n}):0\leq s_{1}\leq\ldots\leq s_{n}\leq1\},
\]
and $E^{\otimes0}=\mathbb{R}$, 
\[
T^{(n)}(E)=\oplus_{i=0}^{n}E^{\otimes i}
\]
and we will say a map $\mathbb{X}:\triangle_{2}\rightarrow T^{(\lfloor p\rfloor)}(E)$
is a multiplicative functional if for all $s\leq u\leq t$, 
\[
\mathbb{X}_{s,u}\otimes\mathbb{X}_{u,t}=\mathbb{X}_{s,t}.
\]
A control is a uniformly continuous function $\omega:\triangle_{2}\rightarrow[0,\infty)$
such that for all $s\leq u\leq t$, 
\[
\omega(s,u)+\omega(u,t)\leq\omega(s,t).
\]
Let $\mathbb{X}^{n}$ denote the projection of $\mathbb{X}$ onto
$E^{\otimes n}$. A $p$-rough path is a multiplicative functional
$\mathbb{X}$ such that there exists a constant $C$ (independent
of time) and a control $\omega$ so that for all $(s,t)\in\triangle_{2}$,
\begin{equation}
\Vert\mathbb{X}_{s,t}^{n}\Vert\leq C\omega(s,t)^{\frac{n}{p}},\,\forall n\leq\lfloor p\rfloor.\label{eq:controlled by}
\end{equation}
If (\ref{eq:controlled by}) holds, we say $\mathbb{X}$ is controlled
by $\omega$. Here and everywhere below the norm $\Vert\cdot\Vert$
can be any norm that is admissible (see Definition 1.25). The readers
may wish to just take $\Vert\cdot\Vert$ to be the projective norm. 
\begin{thm}
\label{thm:(Lyons'-Extension-Theorem}(Lyons' Extension Theorem \cite{Lyons98})
Let $\mathbb{X}:\triangle_{2}\rightarrow T^{(\lfloor p\rfloor)}(E)$
be a $p$-rough path. Suppose further that there exists $\beta\geq p^{2}(1+\sum_{r=3}^{\infty}(\frac{2}{r-2})^{\frac{\lfloor p\rfloor+1}{p}})$
such that 
\begin{equation}
\Vert\mathbb{X}_{s,t}^{n}\Vert\leq\frac{1}{\beta(\frac{n}{p})!}\omega(s,t)^{\frac{n}{p}},\,\forall n\leq\lfloor p\rfloor,\label{eq:lyons factoriall}
\end{equation}
with $(\frac{n}{p})!=\Gamma(\frac{n}{p}+1)$ and $\Gamma$ being the
gamma function. Then there exists a unique extension of $\mathbb{X}$
to a multiplicative functional, which we will also denote as $\mathbb{X}$,
such that $\mathbb{X}$ is also controlled by $\omega$. Moreover,
(\ref{eq:lyons factoriall}) holds for all $n\geq\lfloor p\rfloor+1$. 
\end{thm}
The extended multiplicative functional $\mathbb{X}^{n}$ can be interpreted
as the order $n$ iterated integrals of $\mathbb{X}$. There are several
extensions of this estimate for solutions to differential equations,
see \cite{LyonsYang} and \cite{Factorail Taylor}. The proof of Theorem
\ref{thm:(Lyons'-Extension-Theorem} uses the ``neoclassical inequality''
that for all $a,b\geq0$, 
\[
\sum_{k=0}^{n}\frac{a^{(n-k)\frac{1}{p}}b^{\frac{k}{p}}}{(\frac{n-k}{p})!(\frac{k}{p})!}\leq p\frac{(a+b)^{\frac{n}{p}}}{(\frac{n}{p})!}.
\]
This neoclassical inequality is due to Hino and Hare \cite{HaraHino10},
although there is a slightly less sharp version of this inequality
in Lyons work \cite{Lyons98}. The purpose of this article is to give
an alternative proof of Lyons' estimate (\ref{eq:lyons factoriall})
without using the neoclassical inequality. By focusing on the simpler
case of geometric rough paths, we hope that it will help the readers
in understanding the long computations in \cite{BRP15}. We first
introduce the notion of factorial control. 
\begin{defn}
\label{def:R function}Let $m\leq n$. Then we say a uniformly continuous
function $R:\triangle_{3}\rightarrow[0,\infty)$ is a \textit{factorial
control }if $n\geq m$, 

1. (Control property for $R$) for all $u\leq v\leq s\leq t$, 
\[
R_{u}^{m,n}(v,s)^{\frac{1}{m}}+R_{u}^{m,n}(s,t)^{\frac{1}{m}}\leq R_{u}^{m,n}(v,t)^{\frac{1}{m}}.
\]

2. (Decreasing in $m$) for all $0\leq k\leq m$, 
\[
\frac{1}{(n-m)!}R_{u}^{m,n}(s,t)\leq\frac{c_{m}}{(n-m+k)!}R_{u}^{m-k,n}(s,t).
\]

3. ($R$ has factorial decay) 
\[
\frac{1}{(n-m)!}R_{u}^{m,n}(u,t)\leq\frac{c_{m}\omega(u,t)^{n}}{n!}.
\]

4. ($R$ dominates binomial sum)

\[
\sum_{i=m}^{n}\frac{\omega(u,s)^{n-i}\omega(s,t)^{i}}{(m-i)!i!}\leq\frac{1}{(n-m)!}R_{u}^{m,n}(s,t).
\]

5. (Chen's identity for $R$) 
\[
\sum_{k=1}^{m-1}R_{u}^{m-k,,n-k}(v,s)\frac{\omega(s,t)^{k}}{k!}\leq R_{u}^{m,n}(v,t).
\]

\end{defn}
We will now construct an example of factorial control. Let $\omega$
be a control. Define 
\[
\rho_{u}^{a}(t)=\frac{1}{a}\omega(u,t){}^{a}
\]
and 
\[
S^{(m)}(\rho_{u}^{a}(\cdot))_{s,t}=\int_{s<s_{1}<\ldots<s_{m}<t}\mathrm{d}\rho_{u}^{a}(s_{1})\ldots\mathrm{d}\rho_{u}^{a}(s_{m}).
\]
Let 
\begin{equation}
R_{u}^{m,n}(s,t)=S^{(m)}(\rho_{u}^{\frac{n}{m}}(\cdot))_{s,t}.\label{eq:R function}
\end{equation}

\begin{lem}
The function $R_{u}^{m,n}(s,t)$ defined in (\ref{eq:Dominate by R})
is a factorial decay estimate. \end{lem}
\begin{proof}
Note that the $R$ function has the explicit representation 
\begin{equation}
R_{u}^{m,n}(s,t)=(\frac{m}{n})^{m}\frac{(\omega(u,t)^{\frac{n}{m}}-\omega(u,s)^{\frac{n}{m}})^{m}}{m!}.\label{eq:R explicit}
\end{equation}
This representation gives automatically property 1. for $R$-function.
To show property 2., note that by the inequality that for $a\leq b$
and $\alpha\geq1$ we have 
\[
(a-b)^{\alpha}\leq a^{\alpha}-b^{\alpha}.
\]
Using this with $\alpha=\frac{m}{m-k}$, we have 
\begin{eqnarray*}
R_{u}^{m,n}(s,t) & \leq & (\frac{m}{n})^{m}\frac{1}{m!}(\omega(u,t)^{\frac{n}{m-k}}-\omega(u,s)^{\frac{n}{m-k}})^{m-k}\\
 & = & \frac{m^{m}}{n^{k}(m-k)^{m-k}}\frac{(m-k)!}{m!}R_{u}^{m-k,n}(s,t)\\
 & \leq & \frac{\exp(m)}{n^{k}}R_{u}^{m-k,n}(s,t),
\end{eqnarray*}
where in the final line we used that $m^{m}/m!\leq\exp(m)$. Therefore,
\begin{eqnarray*}
\frac{1}{(n-m)!}R_{u}^{m,n}(s,t) & \leq & \frac{\exp(m)}{n^{k}(n-m)!}R_{u}^{m-k,n}(s,t)\\
 & \leq & \frac{\exp(m)}{(n-m+k)!}R_{u}^{m-k,n}(s,t).
\end{eqnarray*}
For property 3., we see from the explicit representation of $R$ (\ref{eq:R explicit})
that 
\begin{eqnarray*}
R_{u}^{m,n}(u,t) & = & (\frac{m}{n})^{m}\frac{\omega(u,t)^{n}}{m!}.
\end{eqnarray*}
As $m^{m}/m!\leq\exp(m)$, we have 
\[
\frac{1}{(n-m)!}R_{u}^{m,n}(u,t)\leq\exp(m)\frac{\omega(u,t)^{n}}{n!}.
\]
We move on to property 4. Applying Taylor's Theorem with integral
form remainder to $x\rightarrow\frac{x^{n}}{n!}$, we have
\[
\sum_{i=m}^{n}\frac{(y-z){}^{n-i}(x-y){}^{i}}{(n-i)!i!}=\int_{y}^{x}\frac{(a-z)^{n-m}(x-a)^{m-1}}{(n-m)!(m-1)!}\mathrm{d}a.
\]
By reparametrising $a$ as $v\rightarrow z+\omega(u,v)$ and let $x=z+\omega(u,t)$
and $y=z+\omega(u,s)$, we have
\begin{eqnarray*}
 &  & \sum_{i=m}^{n}\frac{\omega(u,s)^{n-i}\big(\omega(u,t)-\omega(u,s)\big)^{i}}{(n-i)!i!}\\
 & = & \int_{s}^{t}\frac{\omega(u,v){}^{n-m}(\omega(u,t)-\omega(u,v))^{m-1}}{(n-m)!(m-1)!}\mathrm{d}\omega(u,v).
\end{eqnarray*}
Therefore, as $\omega$ is a control, 
\begin{eqnarray*}
J & := & \sum_{i=m}^{n}\frac{\omega(u,s)^{n-i}\omega(s,t)^{i}}{(n-i)!i!}\\
 & \leq & \sum_{i=m}^{n}\frac{\omega(u,s)^{n-i}\big(\omega(u,t)-\omega(u,s)\big)^{i}}{(n-i)!i!}\\
 & = & \int_{s}^{t}\frac{\omega(u,v)^{n-m}(\omega(u,t)-\omega(u,v))^{m-1}}{(n-m)!(m-1)!}\mathrm{d}\omega(u,v)\\
 & = & \frac{1}{(n-m)!}\int_{s<s_{1}<\ldots<s_{m}<t}\omega(u,s_{1})^{n-m}\mathrm{d}\omega(u,s_{1})\ldots\mathrm{d}\omega(u,s_{m}).
\end{eqnarray*}
Note that as $s_{1}<\ldots<s_{m}$, 
\[
\omega(u,s_{1})^{n-m}\leq\Pi_{i=1}^{m}\omega(u,s_{i})^{\frac{n-m}{m}}.
\]
Therefore, 
\begin{eqnarray*}
J & \leq & \frac{1}{(n-m)!}\int_{s<s_{1}<\ldots<s_{m}<t}\Pi_{i=1}^{m}\omega(u,s_{i})^{\frac{n-m}{m}}\mathrm{d}\omega(u,s_{i})\\
 & = & \frac{1}{(n-m)!}\int_{s<s_{1}<\ldots<s_{m}<t}\Pi_{i=1}^{m}\mathrm{d}\rho_{u}^{\frac{n}{m}}(s_{i})\\
 & = & \frac{1}{(n-m)!}R_{u}^{m,n}(s,t).
\end{eqnarray*}
To show property 5, we note that as $\omega$ is a control, 
\begin{eqnarray*}
K & := & \sum_{k=1}^{m-1}R_{u}^{m-k,,n-k}(v,s)\frac{\omega(s,t)^{k}}{k!}\\
 & \leq & \sum_{k=1}^{m-1}R_{u}^{m-k,,n-k}(v,s)\frac{(\omega(u,t)-\omega(u,s)){}^{k}}{k!}\\
 & = & \sum_{k=1}^{m-1}\int_{v<s_{1}<\ldots<s_{m-k}<s}\Pi_{i=1}^{m-k}\omega(u,s_{i})^{\frac{n-m}{m-k}}\mathrm{d}\omega(u,s_{1})\ldots\mathrm{d}\omega(u,s_{m-k})\\
 &  & \times\int_{s<s_{m-k+1}<\ldots<s_{m}<t}\mathrm{d}\omega(u,s_{m-k+1})\ldots\mathrm{d}\omega(u,s_{m}).
\end{eqnarray*}
Since 
\[
s_{1}<s_{2}<\ldots<s_{m-k}<s_{m-k+1}<\ldots<s_{m},
\]
we have 
\[
\Pi_{i=1}^{m-k}\omega(u,s_{i})^{\frac{n-m}{m-k}}\leq\Pi_{i=1}^{m}\omega(s,s_{i})^{\frac{n-m}{m}}.
\]
Therefore, 
\begin{eqnarray*}
K & \leq & \sum_{k=1}^{m-1}\int_{v<s_{1}<\ldots<s_{m-k}<s}\Pi_{i=1}^{m-k}\omega(u,s_{i})^{\frac{n-m}{m}}\mathrm{d}\omega(u,s_{i})\\
 &  & \times\int_{s<s_{m-k+1}<\ldots<s_{m}<t}\Pi_{i=m-k+1}^{m}\omega(u,s_{i})^{\frac{n-m}{m}}\mathrm{d}\omega(u,s_{i})\\
 & = & \sum_{k=1}^{m-1}S^{(m-k)}(\rho_{u}^{\frac{n}{m}})_{v,s}S^{(k)}(\rho_{u}^{\frac{n}{m}})_{s,t}.
\end{eqnarray*}
By Chen's identity, 
\[
K\leq S^{(m)}(\rho_{u}^{\frac{n}{m}})_{v,t}=R^{m,n}
\]

\end{proof}
We will use a trick that first appeared in the work of Young \cite{Young}.
This involves carefully choosing a sequence of points to be removed
from a partition and bounding the change in estimate with each removal.
We therefore needs to following definition. 
\begin{defn}
Let $\mathbb{X}:\triangle_{2}\rightarrow T^{\lfloor p\rfloor}(E)$
be a multiplicative functional. If $\mathcal{P}=(t_{0}<t_{1}<\ldots<t_{r})$
is a partition for $[s,t]$, then we define
\[
\mathbb{X}_{s,t}^{n+1,\mathcal{P}}=\sum_{i=0}^{r-1}\sum_{k=1}^{\lfloor p\rfloor}\mathbb{X}_{s,t_{i}}^{n+1-k}\otimes\mathbb{X}_{t_{i},t_{i+1}}^{k}.
\]
\end{defn}
\begin{rem}
Note that we have $\mathbb{X}_{s,t}^{n+1}=\lim_{\max_{i}|t_{i}-t_{i+1}|\rightarrow0}\mathbb{X}_{s,t}^{n+1,\mathcal{P}}$.
\end{rem}
The following algebraic lemma will take care of the algebraic computations
in removing points from a partition. 
\begin{lem}
\label{lem:algebraic lemma }(Algebraic lemma) Let $\mathbb{X}:\triangle_{2}\rightarrow T^{(n)}(E)$
be a multiplicative functional. Then for each $t_{j}$ in the partition
$\mathcal{P}$ of $[s,t]$, 
\begin{eqnarray*}
 &  & \sum_{m\geq\lfloor p\rfloor+1}\mathbb{X}_{u,s}^{n-k}\otimes(\mathbb{X}_{s,t}^{k,\mathcal{P}}-\mathbb{X}_{s,t}^{k,\mathcal{P}\backslash\{t_{j}\}})\\
 & = & \sum_{k=1}^{\lfloor p\rfloor}\sum_{l=\lfloor p\rfloor+1}^{n+1}\mathbb{X}_{u,t_{j-1}}^{n+1-l}\otimes\mathbb{X}_{t_{j-1},t_{j}}^{l-k}\otimes\mathbb{X}_{t_{j},t_{j+1}}^{k}.
\end{eqnarray*}
\end{lem}
\begin{proof}
Suppose we define 
\[
\delta(\mathbb{X}^{n+1})=\mathbb{X}_{s,t}^{n+1,\mathcal{P}}-\mathbb{X}_{s,t}^{n+1,\mathcal{P}\backslash\{t_{j}\}}.
\]
Note that for any $t_{j}\in\mathcal{P}$,
\begin{eqnarray*}
\delta(\mathbb{X}^{n+1}) & = & \sum_{k=1}^{\lfloor p\rfloor}\mathbb{X}_{s,t_{j-1}}^{n+1-k}\otimes\mathbb{X}_{t_{j-1},t_{j}}^{k}+\sum_{k=1}^{\lfloor p\rfloor}\mathbb{X}_{s,t_{j}}^{n+1-k}\otimes\mathbb{X}_{t_{j},t_{j+1}}^{k}\\
 &  & -\sum_{k=1}^{\lfloor p\rfloor}\mathbb{X}_{s,t_{j-1}}^{n+1-k}\otimes\mathbb{X}_{t_{j-1},t_{j+1}}^{k}.
\end{eqnarray*}
 Applying the multiplicative property of $\mathbb{X}_{t_{j-1},t_{j+1}}^{k}$,
we have 
\begin{eqnarray*}
\delta(\mathbb{X}^{n+1}) & = & \sum_{k=1}^{\lfloor p\rfloor}\mathbb{X}_{s,t_{j-1}}^{n+1-k}\otimes\mathbb{X}_{t_{j-1},t_{j}}^{k}+\sum_{k=1}^{\lfloor p\rfloor}\mathbb{X}_{s,t_{j}}^{n+1-k}\otimes\mathbb{X}_{t_{j},t_{j+1}}^{k}\\
 &  & -\sum_{k=1}^{\lfloor p\rfloor}\sum_{l=0}^{k}\mathbb{X}_{s,t_{j-1}}^{n+1-k}\otimes\mathbb{X}_{t_{j-1},t_{j}}^{k-l}\otimes\mathbb{X}_{t_{j},t_{j+1}}^{l}.
\end{eqnarray*}
Note that the term $l=0$ in the third sum would exactly cancel with
the first sum, therefore,
\[
\delta(\mathbb{X}^{n+1})=\sum_{k=1}^{\lfloor p\rfloor}\mathbb{X}_{s,t_{j}}^{n+1-k}\otimes\mathbb{X}_{t_{j},t_{j+1}}^{k}-\sum_{l=1}^{\lfloor p\rfloor}\sum_{k=l}^{\lfloor p\rfloor}\mathbb{X}_{s,t_{j-1}}^{n+1-k}\otimes\mathbb{X}_{t_{j-1},t_{j}}^{k-l}\otimes\mathbb{X}_{t_{j},t_{j+1}}^{l}.
\]
By renaming variable $l$ as $k$, and vice-versa, in the second sum,
we have
\begin{eqnarray}
\delta(\mathbb{X}^{n+1}) & = & \sum_{k=1}^{\lfloor p\rfloor}\mathbb{X}_{s,t_{j}}^{n+1-k}\otimes\mathbb{X}_{t_{j},t_{j+1}}^{k}-\sum_{k=1}^{\lfloor p\rfloor}\sum_{l=k}^{\lfloor p\rfloor}\mathbb{X}_{s,t_{j-1}}^{n+1-l}\otimes\mathbb{X}_{t_{j-1},t_{j}}^{l-k}\otimes\mathbb{X}_{t_{j},t_{j+1}}^{k}\nonumber \\
 & = & \sum_{k=1}^{\lfloor p\rfloor}(\mathbb{X}_{s,t_{j}}^{n+1-k}-\sum_{l=k}^{\lfloor p\rfloor}\mathbb{X}_{s,t_{j-1}}^{n+1-l}\otimes\mathbb{X}_{t_{j-1},t_{j}}^{l-k})\otimes\mathbb{X}_{t_{j},t_{j+1}}^{k}\nonumber \\
 & = & \sum_{k=1}^{\lfloor p\rfloor}\sum_{l=\lfloor p\rfloor+1}^{n+1}\mathbb{X}_{s,t_{j-1}}^{n+1-l}\otimes\mathbb{X}_{t_{j-1},t_{j}}^{l-k}\otimes\mathbb{X}_{t_{j},t_{j+1}}^{k}.\label{eq:computation mark 1}
\end{eqnarray}
Now by (\ref{eq:computation mark 1}), reordering the sum and apply
the multiplicative property once again, we have 
\begin{eqnarray*}
 &  & \sum_{m=\lfloor p\rfloor+1}^{n+1}\mathbb{X}_{u,s}^{n+1-m}\otimes\delta(\mathbb{X}^{m})\\
 & = & \sum_{m=\lfloor p\rfloor+1}^{n+1}\sum_{k=1}^{\lfloor p\rfloor}\sum_{l=\lfloor p\rfloor+1}^{m}\mathbb{X}_{u,s}^{n+1-m}\otimes\mathbb{X}_{s,t_{j-1}}^{m-l}\otimes\mathbb{X}_{t_{j-1},t_{j}}^{l-k}\otimes\mathbb{X}_{t_{j},t_{j+1}}^{k}\\
 & = & \sum_{l=\lfloor p\rfloor+1}^{n+1}\sum_{k=1}^{\lfloor p\rfloor}\sum_{m=l}^{n+1}\mathbb{X}_{u,s}^{n+1-m}\otimes\mathbb{X}_{s,t_{j-1}}^{m-l}\otimes\mathbb{X}_{t_{j-1},t_{j}}^{l-k}\otimes\mathbb{X}_{t_{j},t_{j+1}}^{k}\\
 & = & \sum_{k=1}^{\lfloor p\rfloor}\sum_{l=\lfloor p\rfloor+1}^{n+1}\mathbb{X}_{u,t_{j-1}}^{n+1-l}\otimes\mathbb{X}_{t_{j-1},t_{j}}^{l-k}\otimes\mathbb{X}_{t_{j},t_{j+1}}^{k}.
\end{eqnarray*}

\end{proof}
We now prove our key proposition that will take us within a short
reach of our desired factorial decay estimate.
\begin{prop}
\label{prop:induction}Let $\omega$ be a control and let $R$ be
a corresponding factorial control. Let $\mathbb{X}:\triangle_{2}\rightarrow T^{\lfloor p\rfloor}(E)$
be a $p$-rough path controlled by $\omega$, more precisely, we assume
there exists $\beta$ such that 
\[
\beta\geq\lfloor p\rfloor^{1-\frac{1}{p}}\zeta(\frac{\lfloor p\rfloor+1}{p})2^{\frac{\lfloor p\rfloor+1}{p}}\big[\exp(\lfloor p\rfloor+1)+\lfloor p\rfloor^{1-\frac{1}{p}}c_{p}\big]
\]
where $c_{p}$ is defined in Definition \ref{def:R function} and
\begin{equation}
\Vert\mathbb{X}_{s,t}^{k}\Vert\leq\frac{\omega(s,t)^{k/p}}{\beta(k!)^{\frac{1}{p}}}\,\forall1\leq k\leq\lfloor p\rfloor,\label{eq:factorial 1}
\end{equation}
then for all $m\geq\lfloor p\rfloor$, 
\begin{equation}
\Vert\mathbb{X}_{u,t}^{m}-\sum_{k\leq\lfloor p\rfloor}\mathbb{X}_{u,s}^{m-k}\otimes\mathbb{X}_{s,t}^{k}\Vert\leq\frac{1}{\beta(m-\lfloor p\rfloor-1)!^{\frac{1}{p}}}R_{u}^{\lfloor p\rfloor+1,m}(s,t)^{\frac{1}{p}}.\label{eq:Dominate by R}
\end{equation}
\end{prop}
\begin{rem}
\label{rem:Remark on decreasing} Before embarking on the proof, we
first show that whenever (\ref{eq:Dominate by R}) and (\ref{eq:factorial 1})
holds, we have for all $k\leq\lfloor p\rfloor+1$, 
\[
\Vert\sum_{i\geq\lfloor p\rfloor+1-k}\mathbb{X}_{u,s}^{m-i}\otimes\mathbb{X}_{s,t}^{i}\Vert\leq C_{p}\beta^{-1}R_{u}^{\lfloor p\rfloor+1-k,m}(s,t)^{\frac{1}{p}},
\]
for some constant $C_{p}$ depending only on $p$. First note that
by putting $s=u$ in (\ref{eq:Dominate by R}) and uses property 3.
in Definition \ref{def:R function} that $R$ has factorial decay,
we have
\begin{eqnarray*}
\Vert\mathbb{X}_{u,t}^{m}\Vert & \leq & \frac{1}{\beta\big(m-(\lfloor p\rfloor+1)\big)!^{\frac{1}{p}}}R_{u}^{\lfloor p\rfloor+1,m}(u,t)^{\frac{1}{p}}\\
 & \leq & \frac{c_{p}\omega(u,t)^{m/p}}{\beta(m!)^{1/p}}.
\end{eqnarray*}
Therefore, let $\tilde{c}_{p}=\lfloor p\rfloor^{1-\frac{1}{p}}c_{p}$
and using property 4. in Definition \ref{def:R function} ($R$ dominates
binomial sum), 
\begin{eqnarray*}
\Vert\sum_{i=\lfloor p\rfloor+1-k}^{\lfloor p\rfloor}\mathbb{X}_{u,s}^{m-i}\otimes\mathbb{X}_{s,t}^{i}\Vert & \leq & \beta^{-2}c_{p}\sum_{i=\lfloor p\rfloor+1-k}^{\lfloor p\rfloor}\frac{\omega(u,s)^{(m-i)/p}\omega(s,t)^{i/p}}{(m-i)!^{\frac{1}{p}}(i!)^{\frac{1}{p}}}\\
 & \leq & \beta^{-2}\tilde{c}_{p}\big(\sum_{i=\lfloor p\rfloor+1-k}^{\lfloor p\rfloor}\frac{\omega(u,s)^{m-i}\omega(s,t)^{i}}{(m-i)!i!}\big)^{\frac{1}{p}}\\
 & \leq & \beta^{-2}\tilde{c}_{p}\frac{1}{(m+k-\lfloor p\rfloor-1)!^{\frac{1}{p}}}R_{u}^{\lfloor p\rfloor+1-k,m}(s,t){}^{\frac{1}{p}}.
\end{eqnarray*}
Therefore, by induction 
\begin{eqnarray*}
\Vert\sum_{i=\lfloor p\rfloor+1-k}^{m}\mathbb{X}_{u,s}^{m-i}\otimes\mathbb{X}_{s,t}^{i}\Vert & \leq & \beta^{-1}\frac{1}{(m-\lfloor p\rfloor-1)!^{\frac{1}{p}}}R_{u}^{\lfloor p\rfloor+1,m}(s,t)^{\frac{1}{p}}\\
 &  & +\beta^{-2}\tilde{c}_{p}\frac{1}{(m+k-\lfloor p\rfloor-1)!^{\frac{1}{p}}}R_{u}^{\lfloor p\rfloor+1-k,m}(s,t){}^{\frac{1}{p}}\\
 & \leq & \beta^{-1}C_{p}\frac{1}{(m+k-\lfloor p\rfloor-1)!^{\frac{1}{p}}}R_{u}^{\lfloor p\rfloor+1-k,m}(s,t){}^{\frac{1}{p}},
\end{eqnarray*}
where 
\[
C_{p}=\exp(\lfloor p\rfloor+1)+\lfloor p\rfloor^{1-\frac{1}{p}}c_{p}.
\]
\end{rem}
\begin{proof}
We shall prove the proposition by induction. The base induction step
is trivially true since the left hand side is zero. Assume that (\ref{eq:Dominate by R})
holds for all $m\leq n$.

By the Algebraic Lemma \ref{lem:algebraic lemma }, 
\begin{eqnarray*}
I & := & \Vert\sum_{k\geq\lfloor p\rfloor+1}\mathbb{X}_{u,s}^{n+1-k}\otimes(\mathbb{X}_{s,t}^{k,\mathcal{P}}-\mathbb{X}_{s,t}^{k,\mathcal{P}\backslash\{t_{j}\}})\Vert\\
 & =\Vert & \sum_{k=1}^{\lfloor p\rfloor}\sum_{l=\lfloor p\rfloor+1}^{n+1}\mathbb{X}_{u,t_{j-1}}^{n+1-l}\otimes\mathbb{X}_{t_{j-1},t_{j}}^{l-k}\otimes\mathbb{X}_{t_{j},t_{j+1}}^{k}\Vert\\
 & \leq & \sum_{k=1}^{\lfloor p\rfloor}\Vert\sum_{l=\lfloor p\rfloor+1}^{n+1}\mathbb{X}_{u,t_{j-1}}^{n+1-l}\otimes\mathbb{X}_{t_{j-1},t_{j}}^{l-k}\Vert\Vert\mathbb{X}_{t_{j},t_{j+1}}^{k}\Vert\\
 & = & \sum_{k=1}^{\lfloor p\rfloor}\Vert\sum_{i=\lfloor p\rfloor+1-k}^{n+1-k}\mathbb{X}_{u,t_{j-1}}^{n+1-k-i}\otimes\mathbb{X}_{t_{j-1},t_{j}}^{i}\Vert\Vert\mathbb{X}_{t_{j},t_{j+1}}^{k}\Vert.
\end{eqnarray*}
By (\ref{eq:factorial 1}) and Remark \ref{rem:Remark on decreasing},
\begin{eqnarray*}
I & \leq & \frac{C_{p}}{(n-\lfloor p\rfloor)!^{\frac{1}{p}}}\sum_{k=1}^{\lfloor p\rfloor}\frac{1}{\beta}R_{u}^{\lfloor p\rfloor+1-k,n+1-k}(t_{j-1},t_{j})^{\frac{1}{p}}\frac{\omega(t_{j},t_{j+1})^{k/p}}{\beta(k!)^{1/p}}\\
 & \leq & \frac{\lfloor p\rfloor^{1-\frac{1}{p}}C_{p}}{\beta^{2}(n-\lfloor p\rfloor)!^{\frac{1}{p}}}\big(\sum_{k=1}^{\lfloor p\rfloor}R_{u}^{\lfloor p\rfloor+1-k,n+1-k}(t_{j-1},t_{j})\frac{\omega(t_{j},t_{j+1})^{k}}{k!}\big)^{\frac{1}{p}}.
\end{eqnarray*}
It is here that we use Chen's identity for $R$ function (Property
5 in Definition \ref{def:R function}) to obtain that 
\[
I\leq\frac{\lfloor p\rfloor^{1-\frac{1}{p}}C_{p}}{\beta^{2}(n-\lfloor p\rfloor)!^{\frac{1}{p}}}R_{u}^{\lfloor p\rfloor+1,n+1}(t_{j-1},t_{j+1})^{\frac{1}{p}}.
\]
Since by the control property of factorial control ( property 1. in
Definition \ref{def:R function}), 
\[
\sum_{i=1}^{r-1}R_{u}^{\lfloor p\rfloor+1,n+1}(t_{i-1},t_{i+1})^{\frac{1}{\lfloor p\rfloor+1}}\leq R_{u}^{\lfloor p\rfloor+1,n+1}(s,t)^{\frac{1}{\lfloor p\rfloor+1}},
\]
there exists a $j$ such that 
\[
R_{u}^{\lfloor p\rfloor+1,n+1}(t_{j-1},t_{j+1})^{\frac{1}{\lfloor p\rfloor+1}}\leq\frac{1}{r-1}R_{u}^{\lfloor p\rfloor+1,n+1}(s,t)^{\frac{1}{\lfloor p\rfloor+1}}.
\]
Again by the control property of factorial control ( property 1. in
Definition \ref{def:R function}), 
\[
R_{u}^{\lfloor p\rfloor+1,n+1}(t_{j-1},t_{j+1})\leq R_{u}^{\lfloor p\rfloor+1,n+1}(s,t).
\]
Therefore, 
\[
R_{u}^{\lfloor p\rfloor+1,n+1}(t_{j-1},t_{j+1})^{\frac{1}{\lfloor p\rfloor+1}}\leq(\frac{2}{r-1}\wedge1)R_{u}^{\lfloor p\rfloor+1,n+1}(s,t)^{\frac{1}{\lfloor p\rfloor+1}}.
\]
This gives us that 
\[
I\leq\frac{\lfloor p\rfloor^{1-\frac{1}{p}}C_{p}}{\beta^{2}(n-\lfloor p\rfloor)!^{\frac{1}{p}}}\big(\frac{2}{r-1}\wedge1\big)^{\frac{\lfloor p\rfloor+1}{p}}R_{u}^{\lfloor p\rfloor+1,n+1}(s,t)^{\frac{1}{p}}.
\]
By successively removing points from the partition $\mathcal{P}$,
we have that 
\begin{eqnarray*}
 &  & \Vert\sum_{k\geq\lfloor p\rfloor+1}\mathbb{X}_{u,s}^{n+1-k}\otimes\mathbb{X}_{s,t}^{k,\mathcal{P}}\Vert\\
 & = & \Vert\sum_{k\geq\lfloor p\rfloor+1}\mathbb{X}_{u,s}^{n+1-k}\otimes(\mathbb{X}_{s,t}^{k,\mathcal{P}}-\mathbb{X}_{s,t}^{k,\{s,t\}})\Vert\\
 & \leq & \frac{\lfloor p\rfloor^{1-\frac{1}{p}}C_{p}2^{\frac{\lfloor p\rfloor+1}{p}}}{\beta^{2}(n-\lfloor p\rfloor)!^{\frac{1}{p}}}\zeta(\frac{\lfloor p\rfloor+1}{p})R_{u}^{\lfloor p\rfloor+1,n+1}(s,t)^{\frac{1}{p}}.
\end{eqnarray*}
We may now take $\beta\geq\lfloor p\rfloor^{1-\frac{1}{p}}\zeta(\frac{\lfloor p\rfloor+1}{p})C_{p}2^{\frac{\lfloor p\rfloor+1}{p}}$
and take the partition size $|\mathcal{P}|\rightarrow0$, which gives
that 
\[
\Vert\sum_{k\geq\lfloor p\rfloor+1}\mathbb{X}_{u,s}^{n+1-k}\otimes\mathbb{X}_{s,t}^{k}\Vert\leq\frac{1}{\beta(n-\lfloor p\rfloor)!^{\frac{1}{p}}}R_{u}^{\lfloor p\rfloor+1,n+1}(s,t)^{\frac{1}{p}}.
\]
\end{proof}
\begin{prop}
(Lyons' factorial decay estimate \cite{Lyons98}) Let $\mathbb{X}:\triangle_{2}\rightarrow T^{(\lfloor p\rfloor)}(E)$
be a $p$-rough path controlled by $\omega$,or more precisely, there
exists 
\[
\beta\geq\lfloor p\rfloor^{1-\frac{1}{p}}\zeta(\frac{\lfloor p\rfloor+1}{p})2^{\frac{\lfloor p\rfloor+1}{p}}\big[\exp(\lfloor p\rfloor+1)+\lfloor p\rfloor^{1-\frac{1}{p}}c_{p}\big]
\]
such that
\begin{equation}
\Vert\mathbb{X}_{s,t}^{k}\Vert\leq\frac{\omega(s,t)^{k/p}}{\beta(k!)^{\frac{1}{p}}}\,\forall1\leq k\leq\lfloor p\rfloor,\label{eq:factorial 1-1}
\end{equation}
then for all $m\geq\lfloor p\rfloor+1$,
\[
\Vert\mathbb{X}_{s,t}^{k}\Vert\leq\frac{c_{p}\omega(s,t)^{k/p}}{\beta(k!)^{\frac{1}{p}}},
\]
where the constant $c_{p}$ depends only on $p$ and is defined in
Definition \ref{def:R function}. \end{prop}
\begin{proof}
By Proposition \ref{eq:factorial 1-1} with $u=s$ we have for all
$m\geq\lfloor p\rfloor+1$ 
\begin{eqnarray*}
\Vert\mathbb{X}_{u,t}^{m}\Vert & \leq & \frac{1}{\beta(m-\lfloor p\rfloor-1)!^{\frac{1}{p}}}R_{u}^{\lfloor p\rfloor+1,m}(u,t)^{\frac{1}{p}}\\
 & \leq & \frac{c_{p}\omega(u,t)^{\frac{m}{p}}}{\beta(m!)^{\frac{1}{p}}}.
\end{eqnarray*}
\end{proof}

\end{document}